\documentclass[twoside,12pt,a4]{article}
\usepackage{latexsym}
\usepackage{amsmath,amsthm,amsfonts,amssymb}
\usepackage{graphics,graphicx}
\usepackage{xy} \xyoption{all}
\usepackage{soul}
\usepackage{color}
\usepackage{cancel}
\usepackage{doi}
\usepackage{enumerate}
\usepackage{verbatim}
\usepackage{hyperref}

\def\={&=& }
\def\sobolev#1#2{\lra{#1,#2}_{-\frac12}}

\newtheorem{Thm}{Theorem}[section]

\newtheorem{theorem}[Thm]{Theorem}

\def\Cref#1{Corollary~\ref{#1}}

\def\comm#1#2{ \l[ #1 \, ,\, #2 \r] }

\def\calh{{\mathcal H}}
\def\calt{{\mathcal T}}

\def\cals{{\mathcal S}}
\def\cale{{\mathcal E}}

\def\fa{{\mathfrak A}}
\def\K{{\mathfrak K}}

\def\R{{\mathbb R}}

\def\pslash{\hbox{$\partial$\kern-1.2ex \raise.14ex\hbox{/}\kern.5ex}}
\def\Epslash{\hbox{{\rm {\hbox{$\partial$\kern-1.2ex \raise.14ex\hbox{/}\kern.5ex}\kern-.6ex \lower.28ex \hbox{$_E$}}}}}
\def\Epslashi{\hbox{{\rm {\pslash\kern-.6ex \lower.28ex \hbox{$_{E,i}$}}}}}
\def\Epslashii#1{\hbox{{\rm {\pslash\kern-.6ex \lower.28ex \hbox{$_{E,#1}$}}}}}
\def\Epslasht{\hbox{{\rm {\pslash\kern-.6ex \lower.28ex \hbox{$_{E,\calt}$}}}}}
\def\Epslashtstar{\hbox{{\rm {\pslash\kern-.6ex \lower.28ex \hbox{$_{E,\calt^*}$}}}}}
\def\Epslashti{\hbox{{\rm {\pslash\kern-.6ex \lower.28ex \hbox{$_{E,\calt_i}$}}}}}

\def\hensp#1{\enspace\hbox{#1}\enspace}

\def\l{\left}
\def\r{\right}
\def\la{\lambda}

\def\part{\partial}

\def\be{\begin{equation}}
\def\ee{\end{equation}}
\def\beq{\begin{eqnarray}}
\def\eeq{\end{eqnarray}}
\def\nn{\nonumber \\ }

minus2pt

\font\runningheadfont=cmcsc10

\def\lrp#1{\left( #1\right)}
\def\lrp#1{\left( #1\right)}
\def\lra#1{\left\langle #1\right\rangle}

\def\abs#1{{\l\vert #1 \r\vert}}
\def\norm#1{{\l\Vert #1 \r\Vert}}

\newif\ifMarginNotes \MarginNotestrue
\def\mrgn#1{\ifMarginNotes\setbox0=\vtop{\hsize 6.75pc
   {\noindent\relax #1\par}}\leavevmode
   \vadjust{\dimen0=\dp0 \dimen1=\ht0\advance\dimen1 by .5ex
 \advance\dimen0 by -.5ex
  \kern-\dimen1\hbox{\kern\hsize\kern.5pc$\leftarrow$
  \box0}\kern-\dimen0}\fi}
\catcode`\@=11 \font\twelvemsb=msbm10 scaled 1200 \font\tenmsb=msbm10 \font\ninemsb=msbm7 scaled 1200
\newfam\msbfam
\textfont\msbfam=\twelvemsb  \scriptfont\msbfam=\tenmsb
  \scriptscriptfont\msbfam=\ninemsb
\def\msb@{\hexnumber@\msbfam}
\def\Bbb{\relax\ifmmode\let\next\Bbb@\else
 \def\next{\errmessage{Use \string\Bbb\space only in math
mode}}\fi\next}
\def\Bbb@#1{{\Bbb@@{#1}}}
\def\Bbb@@#1{\fam\msbfam#1}
   \font\twelveeufm=eufm10 scaled 1200 \font\teneufm=eufm10 
\font\seveneufm=eufm7 
\newfam\eufmfam
\textfont\eufmfam=\twelveeufm \scriptfont\eufmfam=\teneufm \scriptscriptfont\eufmfam=\seveneufm
\def\frak{\relax\ifmmode\let\next\frak@\else
 \def\next{\errmessage{Use \string\frak\space only in math mode}}\fi\next}
\def\frak@#1{{\frak@@{#1}}}
\def\frak@@#1{\fam\eufmfam#1}

\catcode`\@=12

\title{Stochastic PDE, Reflection Positivity, and Quantum Fields
}
\author{Arthur Jaffe\\
Harvard University, 
Cambridge, MA 02138, USA
\\
\href{mailto:arthur\_jaffe@harvard.edu}{\color{blue}⟨arthur\_jaffe@harvard.edu⟩\color{black}}
}
\date{\today}
\begin{document} 

\maketitle
 \hsize=7truein \hoffset=-.75truein
\thispagestyle{empty}

\begin{abstract}
We investigate stochastic quantization as a method to go from a classical PDE  (with stochastic time $\lambda$)  to a corresponding quantum theory in the limit $\lambda\to\infty$.  We test the method for a linear PDE satisfied by the free scalar field. 
We begin by giving some background about the importance of  establishing the property of  {\em reflection positivity} for the limit $\lambda\to\infty$. We then prove that the measure determined through stochastic quantization of the free scalar field violates reflection positivity (with respect to reflection of the physical time) for every $\lambda<\infty$.  If a non-linear perturbation of the linear equation is continuous in the perturbation parameter, the same result holds for small perturbations.  
For this reason, one needs to find a modified procedure for stochastic quantization, in order to use that method to obtain a quantum theory. 
\end{abstract}

\tableofcontents

\section{Quantum Theory}\label{Sect:Quantum Theory}
Let $x=(\vec x, t)\in\R^{d}$ denote a space-time point and denote time reflection by the map $\vartheta x=(\vec x,-t)$, and let $\R^{d}_{+}$ denote the subspace with $0\leqslant t$.  Let $\cals'(\R^{d})$ denote the Schwartz distributions on Euclidean space, and suppose that $d\mu(\Phi)$ is a measure on $\cals'(\R^{d})$ with characteristic functional $S(f)=\int e^{i\Phi(f)} d\mu(\Phi)$. Consider the group $G$ of Euclidean transformations of $\R^{d}$, namely rotations, translations, and the reflection $\vartheta$, acting on functions by $f^{g}(x)=f(g^{-1}x)$.  Assume  that $S(f)$ satisfies three properties:
 	\begin{enumerate}
	\item{} Euclidean invariance: $S(f^{g})=S(f)$ for all $g\in G$.
	\item{} Reflection positivity (RP): for any finite set of functions $f_{i}\in\cals(\R^{d}_{+})$, the matrix with elements $S_{ij}=S(\overline{f_{i}}-\vartheta f_{j})$ is positive definite. 
	\item{} Exponential bound: for some $k<\infty$, and some Schwartz space norm $\norm{\ \cdot\ }$, one has  
		\be
			\abs{S(f)}
			\leqslant e^{\norm{f}^{k}}\;.
		\ee
	\end{enumerate}
Then there exists a relativistic quantum theory on a Hilbert space $\calh$ equipped with a unitary representation of the Poincar\'e group.  The resulting Hamiltonian is positive, and it has a Poincar\'e-invariant vacuum vector.  The field satisfies all the Wightman axioms with the possible exception of uniqueness of the vacuum.

These results relate probability theory on function spaces to quantum physics.  Along with the direct proofs of the existence of such measures $d\mu(\Phi)$, they illustrate high points in the extensive development of constructive quantum field theory. See \cite{GJB} for more details and references.   

The main point in establishing the reflection-positivity condition (ii) is that it provides the connection from probability theory to quantum theory.  The positive form determined by (ii) gives the inner product that defines the quantum-mechanical Hilbert space. In the case of a non-Gaussian measure $d\mu(\Phi)$, one always needs to construct this measure as a weak limit of approximating measures. All the standard constructions of  $d\mu(\Phi)$ go to great lengths in order to preserve reflection positivity in the approximations. For once a positivity condition is lost, it becomes very problematic  to establish reflection positivity for the limit.   (This is true in particular for the $\mathcal{P}(\varphi)_{2}$ and $\varphi^{4}_{3}$ examples discussed in \cite{GJB}.)  This fact motivates the investigation of the present paper. 

\subsection{Quantization by SPDE}
An alternative to the approach described in the previous section to obtain $d\mu(\Phi)$ is called {\em stochastic quantization}.  The idea is to study the solution to a classical stochastic partial differential equation (SPDE).  This  idea goes back to an unpublished report of Kurt Symanzik \cite{S}, to the work of Edward Nelson \cite{N1,NB},  and  to Parisi and Wu~\cite{PW} who also observe a relation to super-symmetry.  

However, a full mathematical program using this method to construct a complete, non-linear, relativistic quantum field theory satisfying the Wightman axioms, has not yet been carried out.  
Recently Martin Hairer reinvestigated these questions and has made substantial progress  \cite{H1,H2}, as well as in his many other recent works on the ArXiv.  This includes interesting work on the $\Phi^{4}_{3}$ equation. 

 One wishes to obtain $d\mu(\Phi)$ as the limit $\la\to\infty$ of a sequence of measures $d\mu_{\la}(\Phi)$ defined by a dynamical equation for a field subjected to a random force.   
The  dynamical equation for the classical field $\Phi_{\la}(x)$ is first order in an auxiliary ``stochastic-time'' parameter $\la$.  The  linear driving force $\xi_{\la}(x)$  has a white noise distribution.  One obtains the probability measure $d\mu_{\la}(\Phi)$ from the probability measure $d\nu(\xi)$ on the solutions $\Phi_{\la}$ to the stochastic equation with vanishing initial data.  This measure  is called the stochastic quantization of the equation. 


The general equation for the classical field $\Phi_{\la}(x)$ with a Euclidean action functional $\fa(\Phi_{\la})=\int dx \,\fa(\Phi_{\la}(x))$ is the stochastic partial differential equation (SPDE)
	\be\label{Fundamental SPDE}
		\frac{\partial\Phi_{\la}(x)}{\partial\la}
		=  - \frac12\, \frac{\delta\fa(\Phi_{\la})}{\delta \Phi_{\la}(x)}+ \xi_{\la}(x)\;.
	\ee
Here $ \xi_{\la}(x)$ is the driving force, and $\la$ denotes the auxiliary stochastic ``time'' parameter. One chooses the force $\xi_{\la}(x)$ to be white-noise: this means that $\xi_{\la}(x)$ has the Gaussian probability distribution $d\nu(\xi)$ with mean zero, and with  covariance 
	\be
		\int \xi_{\la}(x)\,\xi_{\la'}(x') \,d\nu(\xi)
		= \delta(\la-\la')\delta(x-x')\;.
	\ee
We must also specify the initial data $\Phi_{0}(x)$ for the solution to \eqref{Stochastic Linear Eqn}.   Here we take zero initial data, $\Phi_{0}(x)=0$.  We see later in the linear case that the initial data vanishes in the solution in the limit $\la\to\infty$.   
	
\subsection{The Measure $d\mu_{\la}(\Phi)$}
One assumes that one can reconstruct the measure $d\mu_{\la}(\Phi)$ from its moments. These moments are the moments in the measure $d\nu(\xi)$ of the solution $\Phi_{\lambda}(x)$ to the classical SPDE \eqref{Fundamental SPDE} with fixed initial data. Define
	\be\label{Defn Measure lambda}
	 	\int \Phi(f)^{n} \,d\mu_{\la} (\Phi)
		= \int \Phi_{\la}(f)^{n} \,d\nu(\xi)\;.
	\ee
As the moments $\int \Phi_{\la}(f_{1})\cdots \Phi_{\la}(f_{n})\,d\nu(\xi)$ are symmetric under permutation of the functions $f_{1}, \ldots, f_{n}$, the non-diagonal moments  $\int \Phi(f_{1})\cdots \Phi(f_{n})\,d\mu_{\la}(\Phi)$ can be obtained from the diagonal moments \eqref{Defn Measure lambda} by polarization.  Thus the moments  \eqref{Defn Measure lambda} determine the measure $d\mu_{\lambda}(\Phi)$. 

In general the measure $d\mu_{\la}(\Phi)$ has a complicated structure and is difficult to study.  However for a linear equation one can easily determine the measure. We now study a linear case in order to understand the relation between $d\mu_{\la}(\Phi)$ and the associated quantum theory.

\subsection{The Linear (Free-Field) Case}
The most elementary example of stochastic quantization arises from the linear equation for the free scalar field.  In this case $\fa(\Phi_{\la}(x))=\frac{1}{2}(\nabla\Phi_{\la}(x))^{2}  +\frac12 m^{2} \Phi_{\la}(x)^{2}   $.  The corresponding  SPDE \eqref{Fundamental SPDE} is 
%
	\be\label{Stochastic Linear Eqn}
		\frac{\partial \Phi_{\la}(x)}{\partial \la}
		= -\frac12\lrp{-\Delta +m^{2}}\,\Phi_{\la}(x) + \xi_{\la}(x)\;.
	\ee
Let $\K_{\la}=e^{-\frac\la2\lrp{-\Delta +m^{2}}}$ denote the heat kernel for the homogeneous equation ($\xi_{\lambda}\equiv0$). Its integral kernel $\K_{\la}(x,x')=\K_{\la}(x-x')$ is translation invariant and  satisfies 
	\be
		\frac{\partial }{\partial \la} \K_{\la}(x-x')
		= \frac12\lrp{\Delta -m^{2}}\K_{\la}(x-x')\;,
	\ee
with initial data 
	\be
		\lim_{\la\to0+}\K_{\la}(x-x') 
		= \delta(x-x')\;.
	\ee
Then the solution to  \eqref{Stochastic Linear Eqn}  is
	$
		\Phi_{\la}(x)
		=  (\K_{\la}\,\Phi_{0})(x)
		+ \int_{0}^{\la} d\alpha\, (\K_{\la-\alpha}\,
				\xi_{\alpha})(x) 
	$.	
Here $\Phi_{0}(x)$ denotes the $\la=0$ initial data. When averaged, 
	\be\label{Inhomogeneous Solution}
	\Phi_{\la}(f)=(\K_{\la}\Phi_{0})(f)+\int_{0}^{\la} d\alpha\, \lrp{\K_{\la-\alpha}\xi_{\alpha}}(f)\;.
	\ee

\subsection{The Measure $d\mu_{\lambda}(\Phi)$ for the Free Field}
Since this solution is linear in $\xi_{\lambda}$, it is clear that the Gaussian distribution of $\xi_{\lambda}$ will yield a Gaussian distribution $d\mu_{\la}(\Phi)$. This measure is determined by its integral and first two moments.  By definition the integral of $d\nu(\xi)$ is one, so from \eqref{Defn Measure lambda}, we infer that $\int d\mu_{\la}(\Phi)=1$.

We claim that the first moment of $d\mu(\Phi_{\la})$ only depends on the initial data for the SPDE and equals
	\be
		\int \Phi_{\la}(f)\,d\nu(\xi)
		= (\K_{\la}\,\Phi_{0})(f)\;.
	\ee
This follows from the fact that the measure $d\nu(\xi)$ has mean zero, and the solution for $\Phi_{\la}$ has the form  \eqref{Inhomogeneous Solution}.   Note $e^{\la \Delta}$ is a contraction on any $L^{p}$, and $0\leqslant\K_{\la}(x-x')$. Hence for any $p$, one infers  $\norm{\K_{\la}\Phi_{0}}_{L^{p}}\to0$  as $\la\to\infty$.  For this reason, we can assume vanishing initial data $\Phi_{0}=0$, without affecting the limit $\la\to\infty$. 

We claim that the second (diagonal) moment of $d\mu_{\la}(\Phi)$ has the form 
	\be
		\int \Phi_{\la}(f)^{2} \,d\nu(\xi)
		= \lrp{ (\K_{\la}\,\Phi_{0})(f)}^{2}
		+ \lra{\bar f, D_{\la}f}_{L^{2}}\;.
	\ee
Here $D_{\la}$ is the covariance of $d\mu_{\la}(\Phi)$.  It is a linear transformation, and in our case it is independent of the initial data.  Hence the initial data does not influence any moment in the limit $\lim_{\la\to\infty}d\mu_{\la}(\Phi)$.   
The form of $D_{\la}$ follows from the solution  \eqref{Inhomogeneous Solution} to the linear SPDE.  We claim that 
	\be
		D_{\la} = (I-e^{-\la C^{-1}})C\;,
		\quad\text{where}\quad
		C=(-\Delta+m^{2})^{-1}\;.	
	\ee
In fact the integral kernel of $D_{\la}$ is 
	\beq\label{Covariance}
		 D_{\la}(x,y)
		&=& \int_{0}^{\la}  d\alpha\, 
\, \int_{0}^{\la}  d\beta\,  \int_{\R^{d}} dx'\int_{\R^{d}} dy' \, \K_{\la-\alpha} (x,x')   \K_{\la-\beta} (y,y') \nn
		&& \qquad \qquad \times \int \xi_{\alpha}(x') \xi_{\beta}(y') 
\,d\nu(\xi) \nn
		&=&\int_{0}^{\la}  d\alpha\, 
\, \int_{0}^{\la}  d\beta\,  \int_{\R^{d}} dx'\int_{\R^{d}} dy' \, \K_{\la-\alpha} (x,x')   \K_{\la-\beta} (y',y) \nn
	&& \qquad \qquad \times \ \delta(\alpha-\beta)\,\delta(x'-y')
\nn&=& \int_{0}^{\la} d\alpha\ 
		\K_{2(\la -\alpha)}(x,y)
		= \lrp{I-e^{-\la C^{-1}}}C(x,y)
		\;.
	\eeq
Here we used the symmetry and the multiplication law for the semigroup $\K_{\la}$.  Also, as $\la\to\infty$,  
	\[
		D_{\la} \nearrow  C\;.
	\]

\setcounter{equation}{0}
\section{Reflection Positivity}
The free relativistic quantum field $\varphi(x)$ is a Wightman field on a Fock-Hilbert space $\calh$.  It arises from the Osterwalder-Schrader quantization of the Gaussian measure $d\mu_{C}(\Phi)$ with characteristic function 
	\be\label{Gaussian Functional}
		S_{C}(f) = e^{-\frac12 \lra{\bar{f},Cf}_{L^{2}}}\;.
	\ee
Here space-time is $d$-dimensional, and one requires $m^{2}>0$ if $d=1,2$.
This field was introduced by Kurt Symanzik \cite{S} as a random field and studied extensively in the free-field case by Edward Nelson \cite{N1}, and later by many others.  It is well-understood that such a random field is equivalent to a classical field acting on a Euclidean Fock space $\cale$ with no-particle state $\Omega^{E}$, see for example \cite{GJB}.  In terms of annihilation and creation operators satisfying $\comm{A(k)}{A(k')^{*}}=\delta(k-k')$, one can represent the classical random field as
	\be
		\Phi(x)
		= \frac{1}{(2\pi)^{d/2}} \int \lrp{A(k)^{*}+A(-k)} 
			\frac{1}{(k^{2}+m^{2})^{1/2}}\,e^{ikx}dk\;.
	\ee
  In this framework one can also write the Gaussian characteristic functional \eqref{Gaussian Functional} as 
	\be
		S_{C}(f) 
		= \lra{\Omega^{E}, e^{i\Phi(f)}\Omega^{E}}_{\cale}
		= \int_{\cals'} e^{i\Phi(f)} d\mu_{C}(\Phi)\;.
	\ee
Konrad Osterwalder and Robert Schrader discovered a more-general framework in 1972, based on the fundamental property of {\em reflection positivity} property \cite{OS1,OS2}. There is a formulation for fermion fields and for gauge fields, as well as for fields of higher spin.  So reflection positivity relates most known quantum theories with corresponding classical ones. 
 This construction is so simple and beautiful, it should be a part of every book on quantum theory.  Unfortunately that must wait for a number of new books to be written!
 
One identifies a time direction $t$ for quantization, and writes $x=(t,\vec x)$.  Let $\vartheta:(t,\vec x)\mapsto (-t,\vec x)$ denote time reflection, and $\Theta$ its push forward to $\cals'(\R^{d})$.  Then RP requires that for $A(\Phi)$ an element of the polynomial algebra $\cale_{+}$ generated by random fields $\Phi(f)$ with $f\in\cals(\R^{d}_{+})$, one has
	\be\label{Measure-RP}
		0\leqslant\lra{A,A}_{\calh} 
		= \lra{A, \Theta A}_{\cale}\;.
	\ee
Let $\mathcal{N}$ denote the null space of this positive form and $\cale_{+}/\mathcal{N}$ the space of equivalence classes differing by a null vector.  The Hilbert space of quantum theory $\calh$ is the completion of the pre-Hilbert space $\cale_{+}/\mathcal{N}$,  in this inner product.  

The vectors in $\calh$ are called the OS quantization of vectors in $\cale_{+}$.  Operators $T$ acting on $\cale_{+}$ and preserving $\mathcal{N}$, also have a quantization $\widehat{T}$  as operators on $\calh$, defined by	$\widehat{T}\widehat{A}=\widehat{TA}$.  This is summarized in the commuting exact diagram of Figure \ref{fig:OS}. 
\begin{figure}[h]
\[
    \xymatrix{
  & 0 \ar[d] & 0 \ar[d] &
   \\
 &  \mathcal{N} \ar[r] ^{T}   \ar[d]_{\text{Id}}&  \mathcal{N} \ar[d]^{\text{Id}}&
   \\
 &   \mathcal{E}_{+} \ar[r] ^{T} \ar[d]_{\wedge}&  \mathcal{E}_{+} \ar[d]^{\wedge}&
  \\
&   \cale_{+}/\mathcal{N} \ar[r] ^{\widehat{T}}  \ar[d]&   \cale_{+}/\mathcal{N} \ar[d]&
     \\
&   0 &  0&
   }
\]
	\caption{OS Quantization of Vectors $[A]\in\cale_{+}/\mathcal{N}\mapsto \widehat A$, and of Operators $T\mapsto \widehat T$.}
	\label{fig:OS}
\end{figure}
\goodbreak

Many families of non-Gaussian measures $d\mu(\Phi)$ on $\cals'(\R^{d})$ that are Euclidean-invariant and reflection-positive are known.  The first examples were shown to exist in space-time of two dimensions,  $d=2$, by Glimm~and~Jaffe \cite{GJ1,GJ2,GJ3,GJ4}, and Glimm, Jaffe, and Spencer \cite{GJS1,GJS2,GJS3,GJS4}.  Additional examples were given by Guerra,~Rosen,~and~Simon \cite{GRS1} and others..

In the more difficult case of  $d=3$ space-time dimensions, the only complete example known is the  $\Phi^{4}_{3}$ theory. Glimm~and~Jaffe proved that  in a finite volume, a reflection-positive measure exists for all couplings~\cite{GJ5}. They showed that one has a convergent sequence of renormalized, approximating action functionals $\mathfrak{A}_{n}$ whose exponentials $e^{-\mathfrak{A}_{n}}$   when multiplied by the standard Gaussian measure $d\mu_{C}(\Phi)$ converge weakly.  But the  limit is inequivalent to the Gaussian. The approximations $\mathfrak{A}_{n}$ are chosen to preserve reflection positivity.  This limit agrees in perturbation theory with the standard perturbation theory in physics texts for $\varphi^{4}_{3}$.  The physics result established in this paper is that the renormalized $\varphi^{4}_{3}$ Hamiltonian $H$ in a finite spatial volume is bounded from below.  Joel Feldman and Osterwalder  combined  the stability result of \cite{GJ5} with a modified version of the cluster expansions for Euclidean fields \cite{GJS1}, to obtain a Euclidean-invariant, reflection-positive measure on $\R^{3}$ for small coupling~\cite{FO}.  

The original stability bound paper~\cite{GJ5} took several years to finish.   In that paper we developed a method to show stability in a region of a cell in phase space of size $O(1)$, and to show independence of different phase space cells, with a quantitative estimate of rapid polynomial decay in terms of a dimensionless distance between cells.   This analysis allowed us to analyze partial expectation of degrees of freedom associated with the phase cells.  

It turned out that the ideas we used overlap a great deal with the ``renormalization group'' methods developed  by Kenneth Wilson~\cite{W1}, which appeared while we were still developing our non-perturbative methods for constructive QFT.  One major difference in  Wilson's approach, and what makes it so appealing, is that his methods are iterative.  Our original methods were inductive, using a somewhat  different method on each length scale. Wilson achieved this simplicity by ignoring effects which appeared to be small.  

While many persons have attempted to reconcile these two methods, much more work needs to be done. In spite of qualitative advances, the conceptually-simpler renormalization group methods have not yet been used to establish the physical clustering properties, that were  proved earlier using the inductive methods. 
The most detailed studies of $\Phi^{4}_{3}$ using the renormalization group methods have been carried out by Brydges, Dimock, and Hurd \cite{BDH1,BDH2}.  David~Moser gave a nice exposition and also some refinements~\cite{M}.  

\setcounter{equation}{0}
\section{The Main Result: $\boldsymbol{d\mu_{\la}}(\Phi)$ is Not Reflection Positive}\label{Sect:RP Fails}
Our main result is that the measure $d\mu_{\la}(\Phi)$ arising from stochastic quantization of the free field is not reflection positive for any stochastic time $\la<\infty$.  
In the case of the free field measure, $\lim_{\la\to\infty}d\mu_{\la}(\Phi)=d\mu_{C}(\Phi)$ exists and is the reflection-positive Gaussian with mean zero and covariance $C=(-\Delta + m^{2})^{-1}$.  
If a non-linear perturbation of the free field is continuous in a perturbation parameter $g$, then the lack of reflection positivity at $\lambda<\infty$ carries over as well to small values of $g$. 

\begin{theorem}\label{Theorem:RPFails}
If $2\leqslant d$ and  $\la<\infty$, $m<\infty$, then RP fails for $d\mu_{\la}(\Phi)$. 
 If $d=1$ and  $1<2\la m^{2}<\infty$, then the RP fails for $d\mu_{\la}(\Phi)$.  RP also fails for toroidal compactification of the space-time in any subset of the  coordinates.
\end{theorem}

Lack of the RP property for $\lambda<\infty$ means that the measure $d\mu_{\la }(\Phi)$ does not yield a  corresponding quantum theory.  
It suggests that one needs to have a better understanding of stochastic quantization in order to pass from a measure on function space to quantum theory.  
One possibility is to find a different distribution for the stochastic driving force in the equation, namely a covariance that preserves reflection positivity for all  $\la<\infty$, or a different equation altogether with that property.  An alternative possibility is to develop a new method to control the $\la\to\infty$ limit, in order to establish positivity in the limit without having a positive approximation.  The aim of either approach would be to understand better how to use stochastic quantization as a mathematical route to establish the existence of a non-linear quantum field theory.    

\begin{proof}[Proof of Theorem \ref{Theorem:RPFails}]  
If $d\mu_{\la}(\Phi)$ is a Gaussian measure with mean zero, then the RP property \eqref{Measure-RP}   with respect to time reflection $\vartheta$  is equivalent to RP for the covariance $D_{\la}$  on $\R^{d}_{+}$, see for example \cite{GJB}.  This means that for each function $f\in L^{2}(\R^{d}_{+})$, supported in the positive-time half-space $\R^{d}_{+}=\R^{d-1}\times \R_{+}$, one has 
	\be\label{CovarianceRP}
		0\leqslant \lra{f, \vartheta Cf}_{L^{2}(\R^{d})}\;.
	\ee
The RP property \eqref{CovarianceRP} for $D_{\la}$ on $\R^{d}_{+}$,  with respect to time reflection $\vartheta$ means: 
This RP property can be verified on $\R^{d}$ by analysis of the Fourier transform.  On subdomains of $\R^{d}$, or in other geometries, it is convenient to note the equivalence between $RP$ for $C$ and the monotonicity of covariance operators $C_{D}\leqslant C_{N}$ with Dirichlet and with Neumann boundary data on the reflection plane \cite{GJ-RP}.   

The RP property for a Gaussian measure with non-zero mean is also equivalent to reflection positivity of the covariance $D_{\la}$ of the measure.  If the mean is 
	\be
	M_{\la}(x)
	= \int \Phi_{\la}(x) \,d\nu(\xi)
	=  (\K_{\la}\,\Phi_{0})(x)\;,
	\ee  
let $\Psi$ have a Gaussian distribution $d\mu_{\la}$ with mean zero and covariance $D_{\la}$. Then  the measure
	\be
		d\mu_{\la}(\Phi) = d\mu_{\la}(\Psi + M_{\la})\;,
	\ee
has mean $M_{\la}(x)$ and covariance $D_{\la}$.  Therefore $d\mu_{\la}(\Phi)$ is reflection positive, if and only if its covariance $D_{\la}$ is reflection positive. 

The RP form defines the pre-inner product on the one-particle space $L^{2}(\R^{d}_{+})$ as 
	\be\label{RPInner}
		\lra{f,f}_{\calh}
		=  \lra{f, \vartheta D_{\la}f}_{L^{2}(\R^{d})}\;.
	\ee 
In order to establish a counterexample to RP for the Gaussian measure, it is only necessary to find one function $f$, supported at positive time, for which 
	\be\label{Gaussian RP Condition}
		 \lra{f, \vartheta \,D_{\la}\,f}_{L^{2}}<0\;.
	\ee
Our strategy is to find a function $f$, in the null-space $\mathcal{N}_{C}$ of the reflection-positive covariance $C$, namely the equilibrium measure, for which \eqref{Gaussian RP Condition} holds.

\subsection{The Case $\boldsymbol{d=1}$
}
Since 
	\be
		\norm{f}_{\calh}^{2} 
		= \lra{f, \vartheta D_{\la}f}_{L^{2}}
		\leqslant \lra{f, D_{\la}f}_{L^{2}}
		 \leqslant \lra{f, Cf}_{L^{2}}
		 =\lra{f,f}_{\mathfrak{H}_{-1}}\;,
	\ee	
the inner product \eqref{RPInner} extends by continuity from $L^{2}$ to the Sobolev space $f\in\mathfrak{H}_{-1}(\R^{d}_{+})$.  	
For $d=1$ the Sobolev space contains the Dirac measure $\delta_{t}$ localized at $0\leqslant t$, so in our example we choose $f$ to be a linear combination of two Dirac delta functions localized at two distinct non-negative times $0\leqslant s<t$. 

We claim that 
	\be
		f(u) = e^{ms}\,\delta_{s}(u) -e^{mt}\,\delta_{t}(u)\in \mathcal{N}_{C}\;.
	\ee
In fact 
	\be
		(\vartheta f)(u) = e^{ms}\,\delta_{s}(-u) -e^{mt}\,\delta_{t}(-u)
		= e^{ms}\,\delta_{-s}(u) -e^{mt}\,\delta_{-t}(u)\;,
	\ee
so 
	\beq\label{Null Vector}
		\lra{ f, \vartheta Cf}_{L^{2}}
		&=& \lra{\vartheta f, Cf}_{L^{2}}
		= \lra{(e^{ms}\delta_{-s}-e^{mt}\delta_{-t}),C(e^{ms}\delta_{s}-e^{mt}\delta_{t})}_{L^{2}}\nn
		&=& e^{2ms}\,\lra{\delta_{-s},C\delta_{s}}
		+e^{2mt}\,\lra{\delta_{-t},C\delta_{t}}\nn
		&&\quad - e^{m(t+s)}\,\lra{\delta_{-s},C\delta_{t}}
		-e^{m(t+s)}\,\lra{\delta_{-t},C\delta_{s}}\nn
		&=&  \frac{1}{2m} \lrp{1+1-1-1}
		=0\;,
	\eeq
and $f\in \mathcal{N}_{C}$ as claimed.  We study 
	\be
		\lra{f, \vartheta D_{\la} f}_{L^{2}} 
		= -e^{-\la m^{2}}
		\lra{f, \vartheta e^{\la \Delta}Cf}_{L^{2}}\;.
	\ee
Then $f$ gives a counterexample to RP for $\la m^{2}<\infty$ in case for some $s,t$,  
	\be\label{NtS}
		0<\lra{\vartheta f, e^{\la \Delta}Cf}_{L^{2}} \;.
	\ee  
%
Define
	\be\label{W}
		W_{\la,m}=  \lrp{4\pi\la}^{1/2}2m\, e^{\la\Delta}\,C\;.
	\ee
Recall that the operator $2mC$ has integral kernel
	\be
		2mC(u,s) = e^{-m\abs{u-s}}\;.
	\ee
Furthermore the operator $\lrp{4\pi\la}^{1/2} e^{\la \Delta}$ has the integral kernel 
	\be
		\lrp{\lrp{4\pi\la}^{1/2}e^{\la \Delta}}(t,u) 
		= \,e^{-\frac{(t-u)^{2}}{4\la}}\;.
	\ee
Thus the integral kernel $W_{\la,m}(t,s)=W_{\la,m}(t-s)$ of $W_{\la,m}$ equals 
	\be\label{W}
		W_{\la,m}(t-s)
		=  \int_{-\infty}^{\infty} du\,  e^{-\frac{(t-s-u)^{2}}{4\la} -m\abs{u} }	
		= W_{\la,m}(s-t)\;.
	\ee
The second equality in \eqref{W} shows that $W_{\la,m}$, which is real, is also  hermitian.   Remark that 
	\be\label{ScaledW}
		W_{\la, m} (t)
		=m^{-1}\,W_{\la m^{2},1}(mt)\;.
	\ee
Thus  without loss of generality we may study $W_{\lambda}(t)=W_{\lambda,1}(t)$.

\subsubsection{A Numerical Check}
We used Mathematica to make a numerical check of whether $f$ violates RP in case $s=0$ and $\la=m=1$.  We study the function  
	\be
		F(t) = \frac12\lrp{W_{1}(0)+e^{2t}\,W_{1}(2t)} -e^{t}\,W_{1}(t)\;,
		\quad\text{with}\quad
		W_{1}(t) = \int_{-\infty}^{\infty} du\,e^{-\frac{(t-u)^{2}}{4}-\abs{u}} \;,
	\ee
If $F(t)>0$ for any positive $t$, then RP fails to hold.   
The Mathematica plot of $F(t)$ appears in Figure \ref{fig:RP-Violated}.  Clearly there are values of $t\in(0,1.5)$ for which $F(t)$ is positive.\footnote{I am grateful to Alex Wozniakowski  for assisting me to use Mathematica to test whether $F(t)$ changes sign.}  So this indicates that RP does not hold for $\Phi_{\la}(t)$ in the measure $d\mu(\Phi_{\la})$.  
\begin{figure}[ht!]
	\centering
	\includegraphics{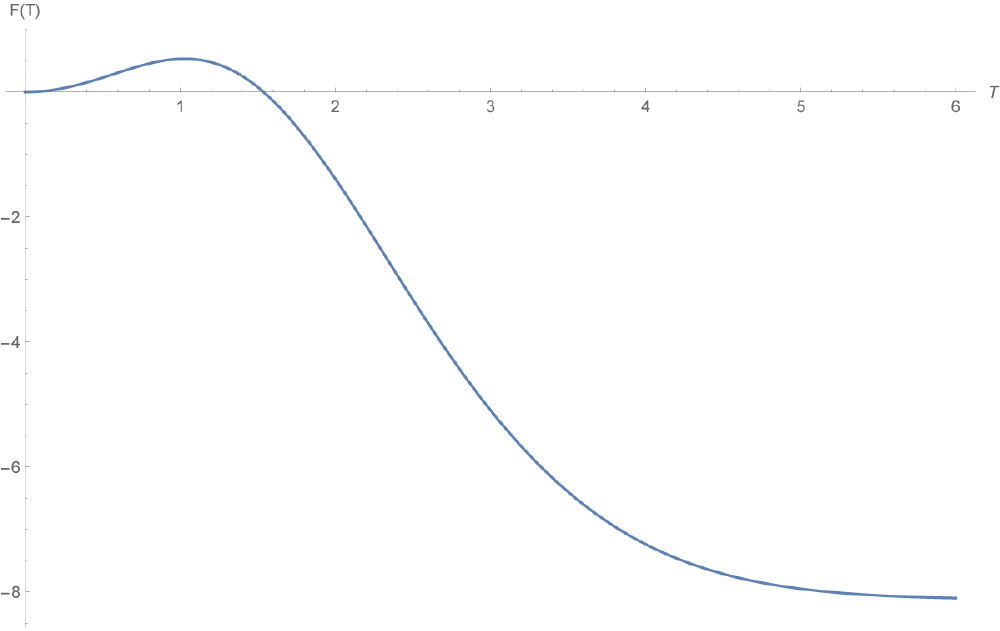}
	\caption{RP of $D_{1}$ requires $F(t)=\frac12\,{W_{1}(0)+ \frac12\,e^{2t}\,W_{1}(2t)}-e^{t}\,W_{1}(t)\leqslant 0$ for all $0\leqslant t$.}
	\label{fig:RP-Violated}
\end{figure}

\subsubsection{The Proof} 
\label{Sect:Counterexample-1}
Since the RP violation occurs for small $t$, we base our proof on expanding $F(t)$ as a power series in $t$. First we extend $W(t)$,  defined for positive $t$, to be an even function at negative $t$. Observe that $F(t)$ is an  analytic function near $t=0$.  Therefore the sign of $F(t)$ for small $0<t$ is determined by the first non-zero term in the power series at $t=0$, and this will be the term of second order.   We show that RP is violated for $\frac12 \leqslant \la m^{2}$.

 From \eqref{W} one has for real $t$, 
	\be
	0<W_{\la}(t) = W_{\la}(0) + \frac{t^{2}}{2} W''_{\la}(0) +O(t^{4})\;,
	\ee
and 
	\be
		 W''_{\la}(0)= c_{\la}  -\frac{1}{2\la} \,W_{\la}(0)\;, \hensp{with}
		 0< c_{\la}=\frac{1}{4\la^{2}}\int_{-\infty}^{\infty} du\, u^{2}e^{-\frac{u^{2}}{4\la}-\abs{u}} \;.
	\ee
Hence
	\be
	F(t) 
	= \frac{ t^{2} }{2}\lrp{W_{\la}(0)+W''_{\la}(0)} +O(t^{3})
	=  \frac{ t^{2} }{2}\lrp{c_{\la}+\lrp{1-\frac{1}{2\la}}W_{\la}(0)} +O(t^{3})\;.
		\ee
As $0<c_{\la}$ and $0<W_{\la}(0)$, the leading non-zero coefficient in $F(t)$ is strictly positive for $\frac12\leqslant\la$. 
Hence RP does not hold for $\frac12\leqslant\la<\infty$.  Reinterpreting this with respect to the scaled function $W_{\la,m}(t)$ according to \eqref{ScaledW}, we infer:  RP fails for $\frac12 \leqslant \la m^{2}<\infty$.

\subsection{The Case $\boldsymbol{1<d}$ and  $\boldsymbol{\la, m^{2}<\infty}$}
\label{Sect:Proof d=1}
For $1<d$  we show that RP for $d\mu_{\la}(\Phi_{\la})$ fails  for all $\la m^{2} \in (0,\infty)$.   Denote the Sobolev-space inner products on $\R^{d}$ and $\R^{d-1}$ respectively as 
	\be
	\lra{f_{1},f_{2}}_{-1}
	= \lra{f_{1},C\,f_{2}}_{L^{2}(\R^{d})}\;,
	\quad\text{and}\quad
	\lra{h_{1},h_{2}}_{-\frac12}
	= \lra{h_{1}, \frac{1}{2\mu}\,h_{2}}_{L^{2}(\R^{d-1})}\;. 
	\ee
In this case $\mu=\sqrt{-\vec\nabla^{2}+m^{2}}$.  We also denote $\mu(\vec p)=\sqrt{\vec p^{\,2}+m^{2}}$ as the multiplication operator in Fourier space given by $\mu$. 

Start by choosing a real, spatial test function $h(\vec x)\in \cals(\R^{d-1})$, whose Fourier transform $\widetilde h(\vec p)$ has compact support.  (Reality only requires $\widetilde{h}(\vec p)=\overline{\widetilde{h}(-\vec p)}$.) Define the family of functions
	\be
		h_{T}(\vec x) 
		= \frac{1}{(2\pi)^{(d-1)/2} }\,\int_{\R^{d-1}} 
		e^{T\mu(\vec p)} \,\widetilde{h}(\vec p) \,e^{i\vec p\cdot \vec x} \,d\vec p\in \cals(\R^{d-1})\;,
	\ee
with 
	\be
		h_{0}=h\;,
		\hensp{and} 0\leqslant T\;.
	\ee
Let $\delta_{S}$ denote the one-dimensional Dirac measure with density $\delta_{S}(t)=\delta(t-S)$. For $0\leqslant S < T$  define  a space-time one-particle function $f(x)=f_{S,T}(x)$ by 
	\be\label{f}
		f
		= h_{S} \otimes \delta_{S}  -  h_{T} \otimes \delta_{T}\;.
	\ee

Note that $f$ is in the null space of the RP form defined by $C$, relative to the time-reflection  $\vartheta$.  In fact 
	\beq\label{d-Test-1}
		\lra{f,\vartheta f}_{-1}
		&=& \lra{f,\vartheta Cf}_{L^{2}}\nn
		&=& \sobolev{h_{S}}{e^{-2S\mu }\,h_{S}}+
		 \sobolev{h_{T}}{e^{-2T\mu}\,h_{T}} \nn
		 && \qquad
		 -  \sobolev{h_{S}}{e^{-(S+T)\mu}\,h_{T}}  
		 -  \sobolev{h_{T}}{e^{-(S+T)\mu}\,h_{S}}  \nn
		 &=&  \sobolev{h}{h} + \sobolev{h}{h}-\sobolev{h}{h}
		 -\sobolev{h}{h}
		 =0\;.
	\eeq
Hence our test of RP relies on whether $\lra{f,f}_{\calh}$ is non-negative, where  
	\be
		\lra{f,f}_{\calh}
		= \lra{f,\vartheta D_{\la}f}_{L^{2}}
		= -  \lra{\vartheta f,  e^{-\la \lrp{-\Delta +m^{2}}}Cf}_{L^{2}}\;.
	\ee
Expanding $f$ according to \eqref{f} yields four terms, each proportional to 
	\beq
		F(t_{1}, t_{2})
		&=& \lra{\lrp{h_{t_{1}}\otimes\delta_{-t_{1}}}, e^{-\la(-\Delta +m^{2})}\,C\,\lrp{ h_{t_{2}} \otimes \delta_{t_{2}} }}_{L^{2}}\nn
		 &=&  \lra{\lrp{ g_{t_{1}}  \otimes \delta_{-t_{1}}}, 
		 X\,\lrp{ g_{t_{2}}  \otimes \delta_{t_{2}}}  }_{L^{2}}\;.
	\eeq
Here  $g_{t}=e^{\frac12 \la \vec\nabla^{2}}\,h_{t}\in\cals(\R^{d-1})$, and  $X=\lrp{e^{-\la(-\frac{\partial^{2}}{\partial t^{2}} +m^{2})}\,C}$.    Note that $\vartheta\delta_{t_{1}}=\delta_{-t_{1}}$, as  $(\vartheta\delta_{t_{1}})(u)=(\vartheta\delta)(u-t_{1})=\delta(-u-t_{1})=\delta(u+t_{1})=\delta_{-t_{1}}(u)$.  However $\vartheta$ does not affect the time in $h_{t_{1}}$. 

The integral kernel for $X$ is real and has the form
	\beq
		X(x,x')
		&=& (4\pi\la)^{-1/2} e^{-\la m^{2}}
			\int_{\R^{4}} du d\vec u \,e^{-\frac{(t-u)^{2}}{4\la}}\,\delta(\vec x-\vec u) 
			\,\lrp{\frac{1}{2\mu}\, e^{-\abs{u-t'}\mu}}(\vec u-\vec x')\nn
		&=&  (4\pi\la)^{-1/2} e^{-\la m^{2}}
			\int_{\R} du 
			\,\lrp{\frac{1}{2\mu}\, e^{-\frac{(t-t'-u)^{2}}{4\la}-\abs{u}\mu}}(\vec x-\vec x')\;.
	\eeq
Thus 
	\be
		{F(t_{1}, t_{2})
		= (4\pi\la)^{-1/2} e^{-\la m^{2}}
		\int_{-\infty}^{\infty} du\,e^{-\frac{(t_{1}+t_{2}-u)^{2}}{4\la}}
			\lra{e^{-\abs{u}\mu}g_{t_{1}}, 
				e^{-\abs{u}\mu}g_{t_{2}}}_{-\frac12}}\;.
	\ee
Since $f$ is real and the kernels are real,  $F(t_{1},t_{2})=F(t_{2},t_{1})$ is real.  Also it is clear from inspection that the compact support of $\widetilde h$ ensures that $F(t_{1},t_{2})$ extends in a neighborhood of $(t_{1},t_{2})=(0,0)$ to a complex analytic function of $(t_{1},t_{2})$ with a convergent power series at the origin. 

Combining these remarks, 
	\be
		\lra{f,f}_{\calh}
		= -\lrp{ F(S,S)+F(T,T)-F(S,T)-F(T,S)}\;.
	\ee
As in \S\ref{Sect:Counterexample-1}, we take $S=0$ and $0<T$. Define $F(T)$ by
	\be\label{Relation to F}
		 \lra{f,f}_{\calh}
		= - (4\pi\la)^{-1/2} e^{-\la m^{2}}  F(T)\;.
	\ee
Then  
	\be
		F(T)= (4\pi\la)^{1/2} e^{\la m^{2}} \lrp{F(0,0) + F(T,T) -2F(0,T)}\;.
	\ee

The function $f$ provides a counterexample to RP if for any $T$ one has both  $0<F(T)$ and $\la,  m^{2}<\infty$.	
Expand $F(T)$ as a power series at $T=0$.  We claim that $F(0)=F'(0)=0$, so 
	\be
		F(T) = \frac{T^{2}}{2}\,F''(0) +O(T^{3})\;.
	\ee
Clearly $F(0)=0$. Also  
	\beq
		F'(T) 
		&=&- \frac{1}{\la}\int_{-\infty}^{\infty} du\,
			(2T-u)
			e^{-\frac{(2T-u)^{2}}{4\la}}
			\lra{e^{-\abs{u}\mu}g_{T}, 
				e^{-\abs{u}\mu}g_{T}}_{-\frac12}\nn
		&&\qquad + 2\int_{-\infty}^{\infty} du\,e^{-\frac{(2T-u)^{2}}{4\la}}
			\lra{e^{-\abs{u}\mu}g_{T}, 
				\mu \,e^{-\abs{u}\mu}g_{T}}_{-\frac12}\nn
		&&\qquad
		+2\,\frac{1}{2\la} \int_{-\infty}^{\infty} du\,
			(T-u)
			e^{-\frac{(T-u)^{2}}{4\la}}
			\lra{e^{-\abs{u}\mu}g_{0}, 
				e^{-\abs{u}\mu}g_{T}}_{-\frac12}\nn
		&&\qquad
		-2 \int_{-\infty}^{\infty} du\,
			e^{-\frac{(T-u)^{2}}{4\la}}
			\lra{e^{-\abs{u}\mu}g_{0}, 
				\mu\,e^{-\abs{u}\mu}g_{T}}_{-\frac12}\;.
	\eeq
Taking $T=0$, the second and last terms cancel, leaving an integrand that is an odd function of  $u$.  Therefore $F'(0)=0$.   Likewise the second derivative equals
	\beq
		F''(T) 
		&=&- \frac{2}{\la}\int_{-\infty}^{\infty} du\,
			e^{-\frac{(2T-u)^{2}}{4\la}}
			\lra{e^{-\abs{u}\mu}g_{T}, 
				e^{-\abs{u}\mu}g_{T}}_{-\frac12}\nn
		&&\qquad +\frac{1}{\la^{2}}\int_{-\infty}^{\infty} du\,
			(2T-u)^{2}
			e^{-\frac{(2T-u)^{2}}{4\la}}
			\lra{e^{-\abs{u}\mu}g_{T}, 
				e^{-\abs{u}\mu}g_{T}}_{-\frac12}\nn
		&& \qquad - \frac{2}{\la}\int_{-\infty}^{\infty} du\,
			(2T-u)
			e^{-\frac{(2T-u)^{2}}{4\la}}
			\lra{e^{-\abs{u}\mu}g_{T},  \mu\,
				e^{-\abs{u}\mu}g_{T}}_{-\frac12} \nn
		&&\qquad - \frac{2}{\la}\int_{-\infty}^{\infty} du\,
		(2T-u)
		e^{-\frac{(2T-u)^{2}}{4\la}}
			\lra{e^{-\abs{u}\mu}g_{T}, 
				\mu \,e^{-\abs{u}\mu}g_{T}}_{-\frac12}\nn
		&&\qquad + 4\int_{-\infty}^{\infty} du\,e^{-\frac{(2T-u)^{2}}{4\la}}
			\lra{e^{-\abs{u}\mu}g_{T}, 
				\mu^{2} \,e^{-\abs{u}\mu}g_{T}}_{-\frac12}\nn
		&&\qquad
		+ \frac{1}{\la} \int_{-\infty}^{\infty} du\,
			e^{-\frac{(T-u)^{2}}{4\la}}
			\lra{e^{-\abs{u}\mu}g_{0}, 
				e^{-\abs{u}\mu}g_{T}}_{-\frac12}\nn
		&&\qquad
		-\frac{1}{2\la^{2}} \int_{-\infty}^{\infty} du\,
			(T-u)^{2}
			e^{-\frac{(T-u)^{2}}{4\la}}
			\lra{e^{-\abs{u}\mu}g_{0}, 
				e^{-\abs{u}\mu}g_{T}}_{-\frac12}\nn
		&&\qquad
		+\frac{1}{\la} \int_{-\infty}^{\infty} du\,
			(T-u)
			e^{-\frac{(T-u)^{2}}{4\la}}
			\lra{e^{-\abs{u}\mu}g_{0}, \mu\,
				e^{-\abs{u}\mu}g_{T}}_{-\frac12}\nn
		&&\qquad
		+\frac{4}{\la} \int_{-\infty}^{\infty} du\,
		(T-u)
			e^{-\frac{(T-u)^{2}}{4\la}}
			\lra{e^{-\abs{u}\mu}g_{0}, 
				\mu\,e^{-\abs{u}\mu}g_{T}}_{-\frac12}\nn
		&&\qquad
		-2 \int_{-\infty}^{\infty} du\,
			e^{-\frac{(T-u)^{2}}{4\la}}
			\lra{e^{-\abs{u}\mu}g_{0}, 
				\mu^{2}\,e^{-\abs{u}\mu}g_{T}}_{-\frac12}\;.
	\eeq
Thus for $T=0$, 
	\beq\label{F''0}
		F''(0)
		&=& \int_{-\infty}^{\infty} du\,e^{-\frac{u^{2}}{4\la}}
			\lra{e^{-\abs{u}\mu}g_{0}, 
			\lrp{2\mu^{2}-\frac{1}{\la} +\frac{u^{2}}{2\la^{2}}}
			\,e^{-\abs{u}\mu}g_{0}}_{-\frac12}\;.
	\eeq
The positivity of $F''(0)$ would be a consequence of the expectation of the operator ${2\mu^{2}-\frac{1}{\la} +\frac{u^{2}}{2\la^{2}}}$ being positive in the vectors $e^{-|u|\mu}\,g_{0}$ under consideration.  

The integral of  the third term, $u^{2}/2\la^{2}$, is  strictly positive for all $\la<\infty$. 
Furthermore $\mu$ acts in Fourier space as multiplication by $\mu(\vec p)$, so $m\leqslant \mu$, and $0\leqslant 2\mu^{2}-\la^{-1}$ if $\frac12\leqslant \la m^{2}$.  This agrees with the conclusion of \S\ref{Sect:Counterexample-1}.
But as $1<d$, we can assume that the support of $\widetilde h$ (which is also the support of $e^{-|u|\mu(\vec p)}\widetilde {g_{0}}$) lies outside the ball of radius $\lrp{2\la}^{-1/2}$.  This entails $\la^{-1}\leqslant2\mu(\vec p)^{2}$ on the support of $\widetilde h$, and  $0\leqslant 2\mu^{2}- \la^{-1}$  on the domain of functions $h$ we consider.  

Therefore we infer for such $h$ that $0<F''(0)$.  Consequently for small, strictly positive $T$, one has $0<F(T)$. Assuming $\la,  m^{2}<\infty$, the relation \eqref{Relation to F} shows that $\lra{f,f}_{\calh}<0$.  Hence we conclude that RP fails in $1<d$ for all $0<\la, m^{2}<\infty$.
\end{proof}

 \end{document}